\documentclass[smallextended]{svjour3}       % onecolumn (second format)
\usepackage{a4wide, amstext, amsmath, amssymb, graphicx, array, epsfig, psfrag, subfigure,todonotes}
\usepackage{mathrsfs}

\newtheorem{assumption}{Assumption} 

\newcommand\bb{\boldsymbol{b}}

\newcommand\bt{\boldsymbol{t}}
\newcommand\calT{\mathscr{T}}
\newcommand\calE{\mathscr{E}}
\newcommand\PP{\mathbb{P}}
\newcommand\rmU{\mathrm{U}}
\newcommand\RR{\mathbb{R}}
\numberwithin{equation}{section}

\date{}
\begin{document}
\title
{Edge-based nonlinear diffusion for finite element approximations of
  convection-diffusion equations and its relation to algebraic flux-correction schemes}
\titlerunning{Nonlinear diffusion and AFC}
\author{Gabriel R. Barrenechea \and Erik Burman \and Fotini Karakatsani}
\institute{G. R. Barrenechea \at  Department of Mathematics and Statistics, 
  University of Strathclyde  26 Richmond Street, Glasgow G1 1XH, United Kingdom, {\tt gabriel.barrenechea@strath.ac.uk}
\and E. Burman \at
              Department of Mathematics, University College London, London, UK-WC1E  6BT, United Kingdom , {\tt e.burman@ucl.ac.uk}
\and F. Karakatsani \at  Department of Mathematics, University of Chester, Thornton Science Park
CH2 4NU, Chester, United Kingdom. {\tt f.karakatsani@chester.ac.uk}
}

\date{Received: date / Accepted: date}

\authorrunning{G.~R.~Barrenechea, E.~Burman, F.~Karakatsani}
\numberwithin{equation}{section} 
\maketitle
%%%%%%%%%%%%%%%%%%%%%%%%%%%%%%%%%%%%%%%%%%%%%%%%%%%%%%%%%%%%%%%%%%%%%%%
\begin{abstract}
For the case of approximation of convection--diffusion equations using
piecewise affine continuous finite elements
a new edge-based nonlinear diffusion operator is proposed that makes
the scheme
satisfy a discrete maximum principle. The diffusion operator is shown
to be Lipschitz continuous and linearity preserving. Using these
properties we provide a full stability and  error analysis, which, in the diffusion dominated
regime, shows existence, uniqueness and optimal convergence. Then
the algebraic flux correction method is recalled and we show that the
present method can be interpreted as an algebraic flux correction
method for a particular definition of the flux limiters. The performance of the method is
illustrated on some numerical test cases in two space dimensions.
\end{abstract}

\keywords{convection--diffusion \and finite element \and discrete
  maximum principle \and nonlinear diffusion \and algebraic flux correction }

 \section{Introduction}

For an open bounded polygonal (polyhedral) domain $\Omega\subseteq\RR^d, d=2,3$, with Lipschitz
boundary, we consider in this work the  steady-state convection-diffusion-reaction equation
\begin{equation}\left\{
\begin{aligned} 
   -\varepsilon\,\Delta u+{\bb}\cdot\nabla u+\sigma\,u=f
&\quad\mbox{in $\Omega$}\, ,\\
    \qquad\qquad u=g &\quad\mbox{on $\partial\Omega$}\,,
 \end{aligned} \right .
 \label{strong-steady}  
\end{equation}
where $\varepsilon>0$ is the diffusion coefficient, 
${\bb}\in L^{\infty}(\Omega)^2$ is a solenoidal convective field, $\sigma>0$ is a real constant, and
$f\in L^2(\Omega)$,  $g\in H^{\frac{1}{2}}(\partial\Omega)$, are given data. In this work we adopt the
standard notation for Sobolev spaces. In particular, for $D\subset {\mathbb R}^d$ we
denote $(\cdot,\cdot)_D$ the $L^2(D)$ (or $L^2(D)^d$) inner product, and by $\|\cdot\|_{l,D}$
($|\cdot|_{l,D}$) the norm (seminorm) in $H^l(D)$ (with the usual convention that $H^0(D)=L^2(D)$).

The weak form of problem (\ref{strong-steady}) is: Find $u\in H^1(\Omega)$ such 
that $u=g$ on $\partial\Omega$ and
\begin{gather}\label{weakform-steady}
   a(u,v) = (f,v)_\Omega\qquad\forall\,v\in H^1_0(\Omega)\,,
\end{gather}
where the bilinear form $a$ is given by
\begin{gather*}
   a(u,v):=\varepsilon\,(\nabla u,\nabla v)_\Omega+({\bb}\cdot\nabla u,v)_\Omega+\sigma(u,v)_\Omega\,.
\end{gather*}

 The weak problem \eqref{weakform-steady} has 
a unique solution $u\in H^1(\Omega)$ and its solution satisfies the following maximum principle (see \cite{GT83}).
\begin{definition}[Maximum Principle]
Assume that  $f\geq 0, \;g\geq 0$ (resp. $\le 0$) and the solution $u$ of \eqref{weakform-steady} is smooth enough. Then, if
$\sigma=0$ and $u$ attains a strict minimum (resp. maximum) at an  interior point $\tilde{x}\in \Omega$, 
then $u$ is constant in $\Omega$. If $\sigma > 0$, then the same conclusion remains valid if we suppose 
in addition that $u(\tilde{x}) < 0$ (resp. $u(\tilde{x})>0$).
\end{definition}

This work deals with the development of a method that satisfies the discrete analogous of the last
definition. The quest for such a method has been a constant for the last couple of decades. Several
methods have been proposed over the years, both in the finite element and finite volume
contexts (see \cite{RST08} for a review). Overall, the common point of all discretisations that
satisfy a discrete maximum principle (DMP) is that they add some diffusion to the equations. This extra diffusion
can lead to a linear method, but in that case it is a well-known fact that a linear method will provide
very diffused numerical solutions, which will converge suboptimally. Due to the previous fact, 
several methods that add nonlinear diffusion have been proposed.

One approach has been to add a so-called shock-capturing term to the
finite element formulation. This typically amounts to a nonlinear
diffusion term where the diffusion coefficient depends nonlinearly on
the finite element residual, making it large in the zones where the
solution is underresolved, but vanish in smooth regions. An analysis
showing that nonlinear
shock capturing methods may lead to a DMP was first proposed in 
\cite{BE02}, and then developed further for the  Laplace operator  in \cite{BE04}, and
for the convection-diffusion equation in \cite{BE05}. For a review of shock capturing methods,
designed to reduce spurious oscillations, without necessarily
satisfying a DMP, see \cite{JK07}. More recent nonlinear discretisations, these
ones based on the idea of blending in order to satisfy the DMP, are the works 
\cite{EG13,BHSISC}, where the emphasis has been given to prove the convergence to an
entropy solution.
Most shock capturing techniques suffer from
the strong nonlinearity introduced when the diffusion coefficient is
made to depend on the finite element residual (and therefore the
gradient of the approximation function). Because of this the analysis 
of such methods is incomplete even when linear model problems with
constant coefficients are considered. In particular, in most cases uniqueness of
solutions can not be proved, and the convergence theory is incomplete.

On the other hand, driven initially by the design of explicit time
stepping schemes for compressible flows, so called Flux Corrected
Transport (FCT) schemes and the
related algebraic flux correction (AFC) schemes were introduced \cite{LMVB88,KT02,Kuzmin06}. 
These schemes 
act on the algebraic level by first modifying the system matrix so
that it has suitable properties to make the system monotonous, while
perturbing the method as little as possible. This crude strategy, however, necessarily
results in a first order scheme. Then a nonlinear switch,
or flux limiter, is introduced making the low order monotone scheme
active only in the zones where the DMP may be violated. These schemes
have also resisted mathematical analysis for a long time, but a number
of results have been proved recently in
\cite{BJK15,BJK15b}. Indeed, in these references, existence of solutions and positivity have been proved, and 
a first error analysis has been performed. Nevertheless, it was shown that the DMP, and even the
convergence of the discrete solution to the continuous one, depend on the geometry of the mesh.

Another approach to combine monotone (low order) finite element methods with
linear diffusion and high order FEM using flux-limiters was proposed very recently in \cite{GNPY14SINUM}. It then
appears that a cross pollination between the idea of AFC and shock-capturing could be fruitful.

The objective of the present paper is to further bridge the gap between the
shock capturing approach and the algebraic flux correction. Indeed we will consider a generalisation of the shock-capturing
term first introduced in \cite{Bu07} to several dimensions, using an
anisotropic diffusion operator along element edges similar to that
introduced in \cite{BE05}. We show that the resulting scheme satisfies the DMP and give an
analysis of the method. In particular we show that the new
shock capturing term is Lipschitz continuous, and, if the mesh is sufficiently regular,  linearity
preserving (see \S~ \ref{properties-dh}), which allows us to
improve greatly on previous results. In \S~\ref{S-existence} we prove existence of solutions, the discrete maximum
principle, and noticeably, uniqueness in the diffusion dominated regime. We then show error estimates, which,
thanks to the combined use of linearity preservation and Lipschitz continuity,
turn out to be optimal  in the diffusion dominated regime, 
for a special  class of meshes (see \S~\ref{error-estimates}).
In \S~\ref{link-AFC}, we revisit the design principles of
AFC and show that the proposed shock-capturing term can be interpreted
as an AFC scheme using a special flux, allowing both for a DMP and
Lipschitz continuity. Some numerical results are finally shown in \S~\ref{numerics}.

\subsection{Notations}

We now introduce some notation that will be needed for the discrete setting.
We consider a family  $\{\calT_h\}_{h>0}$ of shape-regular triangulations of $\Omega$ consisting of
disjoint $d$-simplices $K$. We define $h_K:={\rm diam} (K)$, and $h=\max\{ h_K:K\in\calT_h\}$.
We associate with the triangulation $\calT_h$ the finite element spaces  
\begin{equation}
\mathcal{V}_h:=\{\chi \in H^1({\Omega}):  \chi|_K\in \PP_1(K)\, \forall K\in\calT_h \} ,
 \quad\text{and} \quad \mathcal{V}_h^0:= \mathcal{V}_h \cap H_0^1(\Omega),
\end{equation}
where $\PP_\ell(D)$ is the space of  polynomials of degree at most $\ell$ on $D.$ The nodes
of $\calT_h$ are denoted by $\{x_i\}_{i=1}^N$, and the usual associated basis functions of
$\mathcal{V}_h$ are denoted by $\{\psi_i\}_{i=1}^N$. 

We let $\calE_h$  be the set of the interior edges of $\calT_h$. For every edge
$E\in\calE_h$,  we define $h_E:=|E|$ and $\omega_E:=\{ K\in\calT_h:K\cap E\not=\emptyset\}$, and
 fix one unit tangent vector, denoted by $\bt$.

For an interior node $x_i$, we define the neighborhoods $\calE_i:= \{E \in \calE_h: x_i\in E\}$,  
$\Omega_i:=\{K\in \calT_h:  x_i\in K \}$, and the set
\begin{equation}
S_i:=\{ j\in \{1,\ldots,N\}\setminus\{ i\}\;:\;x_j\;\textrm{shares an internal edge with}\; x_i\}\,.
\end{equation}

Finally, we will say that the triangulation $\calT_h$ is {\it symmetric with respect to its internal nodes}
if for every internal node $x_i$ the following holds: for all $j\in S_i$ there exists $k\in S_i$ such that
$x_j-x_i=-(x_k-x_i)$ (see Figure \ref{fig-mesh} for examples in two space dimensions).

\begin{figure}[h!tb]
\begin{center}
\subfigure[]{\includegraphics[width=0.25\textwidth]{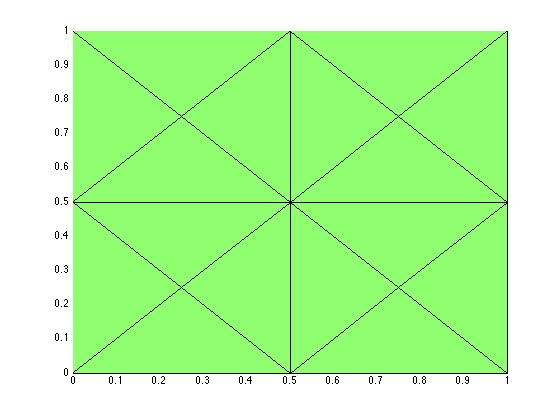}}\hspace*{1ex}
\subfigure[]{\includegraphics[width=0.25\textwidth]{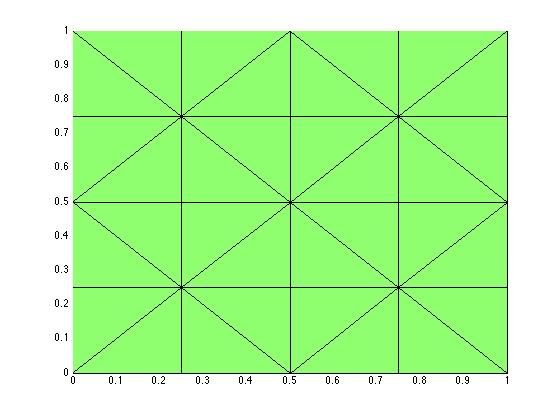}}\hspace*{1ex}
\subfigure[]{\includegraphics[width=0.25\textwidth]{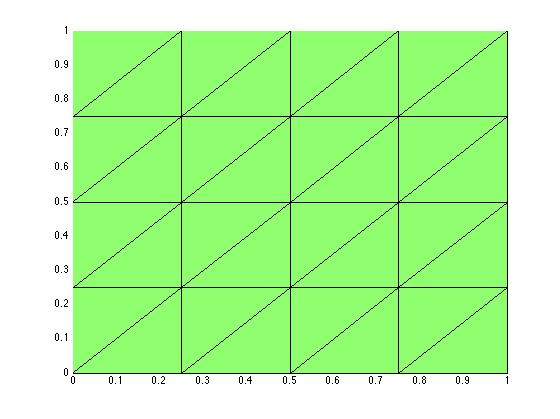}}\hspace*{1ex}
\subfigure[]{\includegraphics[width=0.25\textwidth]{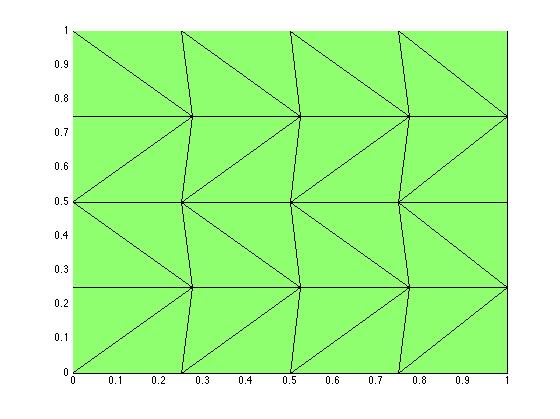}}\hspace*{1ex}
\end{center}
\caption{In two dimensions, meshes (a)-(c) are examples of symmetric  meshes. Mesh (d) is a non-symmetric, non-Delaunay mesh.}\label{fig-mesh}
\end{figure}

\section{The nonlinear discretisation}
The standard finite element method for the problem
\eqref{weakform-steady} takes the form: 
 Find $u_h\in \mathcal{V}_h$ such that $u_h-u_{bh} \in \mathcal{V}_h^0$ and
\begin{equation}\label{FEM}
 a(u_h,v_h) = (f,v_h)_\Omega\qquad
   \forall\,\,v_h\in \mathcal{V}_h^0\,.
\end{equation}
Here,  $u_{bh}\in\mathcal{V}_h$ is introduced to approximate the boundary condition $g$.
Then, we propose  the following stabilised method to discretise \eqref{weakform-steady}: Find 
 $u_h\in \mathcal{V}_h$ such that $u_h-u_{bh} \in \mathcal{V}_h^0$ and
\begin{equation}\label{LSC}
  \tilde{a}(u_h;v_h):= a(u_h,v_h)+d_h(u_h; u_h, v_h) = (f,v_h)_\Omega\qquad
   \forall\,\,v_h\in \mathcal{V}_h^0\,.
\end{equation}
The stabilisation
term $d_h(\cdot \; ; \cdot, \; \cdot)$ is defined by
\begin{equation}
   d_h(w_h ; u_h, v_h)
             =\sum_{E\in \calE_h}\,  \gamma_0^{}\,  h_E^d \, \alpha_E^{}(w_h)( \partial_{\bt}u_h, \partial_{\bt}v_h)_E\,, \end{equation} 
where $\gamma_0^{}>0$, and $\alpha_E^{}:\mathcal{V}_h\to [0,1]$ is a  constant function 
on each $E$ defined  by 
\begin{equation}
\alpha_E^{}(w_h):= \max_{x\in E} \big[ \xi_{w_h}(x)\big]^p\quad,\quad p\in \; [1, +\infty)\,,
\end{equation}
with $\xi_{w_h} \in \mathcal{V}_h$ given by
\begin{equation}
\xi_{w_h} (x_i) := \left\{
\begin{aligned} 
 \frac{\left|\sum_{j\in S_i} w_h(x_i)-w_h(x_j)\right|}{\sum_{j\in S_i}|w_h(x_i)-w_h(x_j)|},&\quad\mbox{if}\;\;
\sum_{j\in S_i} |w_h(x_i)-w_h(x_j)| \neq 0\, ,\\
    \qquad\qquad 0, &\quad\mbox{otherwise}\,.
 \end{aligned} \right .
\end{equation}
The value for $p$ will determine the rate of decay of the numerical
diffusion with the distance to the critical points. A value closer to
$1$ will add more diffusion in the far field, while a larger value
will make the diffusion vanish faster,
but on the other hand, increasing $p$ may make the nonlinear system
more difficult to solve. In principle, as $p$ goes to infinity the
method will add the perturbations only in points with local extrema. In our calculations
we have tested several different values for $p$, and have presented those for $p=1$ and $p=4$, the
latter giving better numerical results, while keeping the nonlinear solver converging within a reasonable
number of iterations.

\begin{remark} All the proofs presented in this work are valid in both the two and three-dimensional
cases. Just to simplify the presentation we will restrict ourselves from now on to the two-dimensional case $d=2$.
\end{remark}

\subsection{Properties of $d_h(\cdot;\cdot,\cdot)$}\label{properties-dh}
We start noticing that 
\[ \sum_{j\in S_i} |w_h(x_i) - w_h(x_j)| = 0\Longrightarrow w_h|_{\Omega_i}=c\in\mathbb{R}\,.\]
 This prevents the method from adding artificial diffusion to the equations in regions in which the solution
is constant. Moreover, the method is as well linearity preserving if the mesh is symmetric with respect
 to its interior nodes. In fact, if  $E\in\calE_h$ has
endpoints $x_i$ and $x_j$,  and $v_h\in\PP_1(\omega_E)$, then
\begin{equation}
\sum_{l\in S_i} v_h(x_i)-v_h(x_l)=0\quad\textrm{and}\quad
\sum_{l\in S_j} v_h(x_j)-v_h(x_l)=0\,,
\end{equation}
which gives $\alpha_E^{}(v_h)=0$. Then, the method does not add extra diffusion in smooth regions,
whenever the mesh is sufficiently structured.

The next step is to show that $d_h(\cdot;\cdot,\cdot)$ is continuous. More precisely,
it is Lipschitz continuous, and the next result is the first step towards this.

\begin{lemma}\label{Lem-Lip-1}
For  any $v_h, w_h\in \mathcal{V}_h$, and any given internal node $x_i$, the following holds
\begin{equation}
|\xi_{v_h}(x_i)-\xi_{w_h}(x_i)|\le 4\frac{\sum_{E\in\calE_i}h_E|\partial_{\bt}(v_h-w_h)|}{
\sum_{E\in\calE_i}h_E\big(|\partial_{\bt} v_h|+|\partial_{\bt} w_h|\big)}\,.
\end{equation}
\end{lemma}

\begin{proof} It is enough to suppose that $\sum_{j\in S_i}|v_h(x_i)-v_h(x_j)|>0$ and 
$\sum_{j\in S_i}|w_h(x_i)-w_h(x_j)|>0$, otherwise the claim is obvious.  A quick calculation gives
\begin{align*}
&|\xi_{v_h}(x_i)-\xi_{w_h}(x_i)|= \left|\frac{\left|\sum_{j\in S_i}v_h(x_i)-v_h(x_j)\right|}{\sum_{E\in\calE_i}h_E
|\partial_{\bt}v_h|}
- \frac{\left|\sum_{j\in S_i}w_h(x_i)-w_h(x_j)\right|}{\sum_{E\in\calE_i}h_E|\partial_{\bt}w_h|}\right|\\
&\le \left|\frac{\left|\sum_{j\in S_i}v_h(x_i)-v_h(x_j)\right|-\left|\sum_{j\in S_i} w_h(x_i)-w_h(x_j)\right|}{\sum_{E\in\calE_i} h_E|\partial_{\bt}v_h|}\right|\\
&\quad+
\left|\sum_{j\in S_i}w_h(x_i)-w_h(x_j)\right|\left|
\frac{1}{\sum_{E\in\calE_i}h_E|\partial_{\bt}v_h|}-\frac{1}{\sum_{E\in\calE_i}h_E|\partial_{\bt}w_h|}\right|\\
&\le  \frac{\sum_{E\in\calE_i}h_E|\partial_{\bt}(v_h-w_h)|}{\sum_{E\in\calE_i}h_E|\partial_{\bt}v_h|}+
\frac{\left|\sum_{j\in S_i}w_h(x_i)-w_h(x_j)\right|\,\left|\sum_{E\in\calE_i}h_E\left(|\partial_{\bt}w_h|-|\partial_{\bt}v_h|\right)\right|}
{\sum_{E\in\calE_i}h_E|\partial_{\bt}v_h|\sum_{E\in\calE_i}h_E|\partial_{\bt}w_h|}\\
&\le 2\,\frac{\sum_{E\in\calE_i}h_E|\partial_{\bt}(v_h-w_h)|}{\sum_{E\in\calE_i}h_E|\partial_{\bt}v_h|}\,.
\end{align*}
The following estimate can be proved in an analogous way
\begin{equation*}
|\xi_{v_h}(x_i)-\xi_{w_h}(x_i)|\le 2\,\frac{\sum_{E\in\calE_i}h_E|\partial_{\bt}(v_h-w_h)|}{\sum_{E\in\calE_i}h_E|\partial_{\bt}w_h|}\,.
\end{equation*}
Then,
\begin{equation}
|\xi_{v_h}(x_i)-\xi_{w_h}(x_i)|\le 2\min\left\{
\frac{1}{\sum_{E\in\calE_i}h_E|\partial_{\bt}v_h|}, \frac{1}{\sum_{E\in\calE_i}h_E|\partial_{\bt}w_h|}\right\}\,
\sum_{E\in\calE_i}h_E|\partial_{\bt}(v_h-w_h)|\,,
\end{equation}
which gives the desired result upon applying the estimate $\min\{a^{-1},b^{-1}\} \le \frac{2}{a+b}$,
for two positive numbers $a$ and $b$. \qed
\end{proof}

The Lipschitz continuity of $d_h(\cdot;\cdot,\cdot)$ appears then as a consequence of the previous result.

\begin{lemma}\label{Lemma-Lip-2} The nonlinear form $d_h(\cdot;\cdot,\cdot)$ is Lipschitz continuous. More precisely,
there exists $C_{\rm lip}>0$, independent of $h$, such that, for all $v_h,w_h,z_h\in \mathcal{V}_h$, the following holds
\begin{equation}
|d_h(v_h;v_h,z_h)-d_h(w_h;w_h,z_h)|\le C_{\rm lip}\gamma_0^{}h\,|v_h-w_h|_{1,\Omega}\,|z_h|_{1,\Omega}\,.
\end{equation}
\end{lemma}

\begin{proof} We have
\begin{align}
& d_h(v_h;v_h,z_h)-d_h(w_h;w_h,z_h)=\sum_{E\in\calE_h}\gamma_0^{}h_E^2(\alpha_E^{}(v_h)
\partial_{\bt}v_h-\alpha_E^{}(w_h)\partial_{\bt}w_h, \partial_{\bt} z_h)_E \nonumber\\
&=\sum_{E\in\calE_h}\gamma_0^{}h_E^2\alpha_E^{}(v_h)(\partial_{\bt}v_h-\partial_{\bt}w_h,\partial_{\bt}z_h)_E 
+ \gamma_0^{}h_E^2(\alpha_E^{}(v_h)-\alpha_E^{}(w_h))(\partial_{\bt}w_h, \partial_{\bt} z_h)_E\,. \label{Lip-1}
\end{align}
The first term in the above estimate is bounded using the fact that $|\alpha_E^{}(v_h)|\le 1$, the
Cauchy-Schwarz inequality, a local trace inequality, and the mesh regularity, to give
\begin{equation}\label{Lip-2}
\sum_{E\in\calE_h}\gamma_0^{}h_E^2\alpha_E^{}(v_h)(\partial_{\bt}v_h-\partial_{\bt}w_h,\partial_{\bt}z_h)_E\le
C\gamma_0^{}h\,|v_h-w_h|_{1,\Omega}|z_h|_{1,\Omega}\,.
\end{equation}
The second term is bounded next. For this, a general edge $E\in\calE_h$ will be considered as having
$x_i$ and $x_j$ as endpoints, where $x_i$
  is chosen to be the vertex such that $\alpha_E(v_h)=
  \xi^p_{v_h}(x_i)$. We then divide $\calE_h=E_1\cup E_2$, where 
\begin{align*}
E_1:=\{ E\in\calE_h: \alpha_E^{}(v_h) =\xi_{v_h}^p(x_i), \alpha_E^{}(w_h)=\xi_{w_h}^p(x_i)\}\,,\\
E_2:=\{ E\in\calE_h: \alpha_E^{}(v_h) =\xi_{v_h}^p(x_i), \alpha_E^{}(w_h)=\xi_{w_h}^p(x_j)\}\,,
\end{align*}
and the second term in \eqref{Lip-1} reduces to
\begin{equation*}
\sum_{E\in E_1}\gamma_0^{}h_E^2\big((\xi_{v_h}^p(x_i)-\xi_{w_h}^p(x_i))\partial_{\bt}w_h, \partial_{\bt} z_h\big)_E
+\sum_{E\in E_2} \gamma_0^{}h_E^2\big((\xi_{v_h}^p(x_i)-\xi_{w_h}^p(x_j))\partial_{\bt}w_h, \partial_{\bt} z_h\big)_E\,.
\end{equation*}
We now remark that for two numbers $a,b\in [0,1]$ we have
\[
|a^p-b^p| = |a-b|\sum_{l=0}^{p-1}a^lb^{p-1-l}\le p\,|a-b|\,,
\]
and the term in $E_1$ is bounded using Lemma \ref{Lem-Lip-1}. In fact, from the mesh regularity  
there exists $C>0$, independent of $h$, such that for all $E,F\in \calE_i, h_F\le C h_E$.
Moreover, the number of edges in $\calE_i$ is uniformly bounded, independently of $h$. Then,
using Cauchy-Schwarz's inequality and a local trace inequality we arrive at
\begin{align}
&\sum_{E\in E_1}\gamma_0^{}h_E^2\big((\xi_{v_h}^p(x_i)-\xi_{w_h}^p(x_i))\partial_{\bt}w_h, \partial_{\bt} z_h\big)_E \nonumber\\
&\le p\sum_{E\in E_1}\gamma_0^{}h_E^2\big(|\xi_{v_h}(x_i)-\xi_{w_h}(x_i)|\partial_{\bt}w_h, \partial_{\bt} z_h\big)_E \nonumber\\
&\le p\sum_{E\in E_1}\gamma_0^{}h_E^2\left(4\frac{\sum_{F\in\calE_i}h_F|\partial_{\bt}(v_h-w_h)|_{F}|}{\sum_{F\in\calE_i}
h_F(|\partial_{\bt} v_h|_F|+|\partial_{\bt} w_h|_F|)}|\partial_{\bt}w_h|,|\partial_{\bt}z_h|\right)_E  \nonumber\\
& \le 4p\,\gamma_0^{}\sum_{E\in E_1}h_E^2\left( \sum_{F\in\calE_i}\big|\partial_{\bt}(v_h-w_h)|_F^{}\big|,|\partial_{\bt}z_h|\right)_E  \nonumber\\
&\le C\gamma_0^{}h\,|v_h-w_h|_{1,\Omega}|z_h|_{1,\Omega}\,. \label{Lip-3}
\end{align}
The sum over $E_2$ is bounded next. First, using \eqref{Lip-3} we get
\begin{align*}
&\sum_{E\in E_2}
\gamma_0^{}h_E^2\big((\xi_{v_h}^p(x_i)-\xi_{w_h}^p(x_j))\partial_{\bt}w_h, \partial_{\bt}
z_h\big)_E \\
&= \sum_{E\in E_2} \gamma_0^{}h_E^2\big((\xi_{v_h}^p(x_i)-\xi_{w_h}^p(x_i))\partial_{\bt}w_h, \partial_{\bt} z_h\big)_E\\
&+
\sum_{E\in E_2} \gamma_0^{}h_E^2\big((\xi_{w_h}^p(x_i)-\xi_{w_h}^p(x_j))\partial_{\bt}w_h, 
\partial_{\bt} z_h\big)_E\\
&\le C\gamma_0^{}h\,|v_h-w_h|_{1,\Omega}|z_h|_{1,\Omega}+\sum_{E\in E_2} \gamma_0^{}h_E^2
\big((\xi_{w_h}^p(x_i)-\xi_{w_h}^p(x_j))\partial_{\bt}w_h, \partial_{\bt} z_h\big)_E\,.
\end{align*}
In an analogous way we obtain
\begin{align*}
&\sum_{E\in E_2} \gamma_0^{}h_E^2\big((\xi_{v_h}^p(x_i)-\xi_{w_h}^p(x_j))\partial_{\bt}w_h, \partial_{\bt} z_h\big)_E
\le C\gamma_0^{}\,h\,|v_h-w_h|_{1,\Omega}|z_h|_{1,\Omega}\\&\quad +\sum_{E\in E_2} \gamma_0^{}h_E^2\big((\xi_{v_h}^p(x_i)-\xi_{v_h}^p(x_j))\partial_{\bt}w_h, \partial_{\bt} z_h\big)_E\,.
\end{align*}
Hence
\begin{align}
&\sum_{E\in E_2} \gamma_0^{}h_E^2\big((\xi_{v_h}^p(x_i)-\xi_{w_h}^p(x_j))\partial_{\bt}w_h, \partial_{\bt} z_h\big)_E\le
C\gamma_0^{}h\,|v_h-w_h|_{1,\Omega}|z_h|_{1,\Omega} \,+ \nonumber\\
&\sum_{E\in E_2}\gamma_0^{}h_E^2\min\{ (\xi_{v_h}^p(x_i)-\xi_{v_h}^p(x_j))\big(\partial_{\bt}w_h, \partial_{\bt} z_h\big)_E,
(\xi_{w_h}^p(x_i)-\xi_{w_h}^p(x_j))\big(\partial_{\bt}w_h, \partial_{\bt} z_h\big)_E\} \nonumber\\
&\le C\gamma_0^{}h\,|v_h-w_h|_{1,\Omega}|z_h|_{1,\Omega}\,, \label{Lip-4}
\end{align}
since the last term in the middle inequality is always non-positive, since by construction, for $E \in E_2$, $\xi_{v_h}^p(x_i) -
\xi_{v_h}^p(x_j)\ge 0$ and $\xi_{w_h}^p(x_i) -
\xi_{w_h}^p(x_j)\le 0$. The result then follows collecting \eqref{Lip-1}-\eqref{Lip-4}. \qed
\end{proof}

\subsection{Solvability of the discrete problem}\label{S-existence}
This section is devoted to analyse the existence of solutions for \eqref{LSC}.  It is interesting to remark that, 
thanks to the Lipschitz continuity of $d_h(\cdot;\cdot,\cdot)$, the solution can be proved to be unique in the diffusion-dominated regime.

\begin{lemma}\label{exis-1}
Let   $T_h: \mathcal{V}_h^0 \rightarrow \mathcal{V}_h^0$ be the operator defined by
\begin{equation}
[T_h z_h, v_h] = \,\tilde{a}(z_h + u_{bh} ; v_h)  - (f,v_h)_\Omega, \quad z_h, v_h \in \mathcal{V}_h^0\,.           
\end{equation}
Then,
 \begin{equation}
[T_h z_h, z_h] \ge c_1 |z_h|_{1,\Omega}^2  -  c_2 (\|u_{bh}\|_{1,{\Omega}}^2 + \|f\|_{0,{\Omega}}^2),
\label{T_h_property}
\end{equation} 
where $c_1, c_2$ are positive constants independent of $z_h, f$,  and $g$.
\end{lemma}
\begin{proof}  For this proof only, we will consider constants $C>0$ that may depend on the 
 the physical coefficients. From the definition of $a$ it follows
that
\begin{equation}
a(z_h,z_h)=\,\varepsilon\,|z_h|^2_{1,\Omega}+(\sigma z_h,z_h)\ge \varepsilon\,|z_h|^2_{1,\Omega}\,.
\end{equation}
Moreover, the definition of $d_h(\cdot;\cdot,\cdot)$ and the fact that $0\le  \alpha_E^{}(z_h + u_{bh})$ give
\begin{equation}
d_h(z_h + u_{bh} ; z_h, z_h) = 
\sum_{E\in \calE_h}\,\gamma_0^{} \, h_E^2\, \alpha_E^{}(z_h + u_{bh}) \|\partial_{\bt} z_h \|_{0,E}^2 \geq 0.
 \end{equation}
Then,  the definition of the operator $T_h$ gives
\begin{gather}
[T_h z_h, z_h] 
\ge \varepsilon |z_h|_{1,\Omega}^2   + a(u_{bh}, z_h) +
d_h(z_h + u_{bh} ; u_{bh}, z_h) - (f,z_h)_\Omega\,.
\end{gather}
Next, the Cauchy-Schwarz  and Poincar\'e inequalities lead to the following bound
 \begin{align}
&| a(u_{bh}, z_h)| \,=\, \left|\varepsilon (\nabla u_{bh},  \nabla z_h)_\Omega + (\bb\cdot \nabla u_{bh}, z_h)_\Omega
+(\sigma u_{bh},z_h)^{}\right| \nonumber \\ 
&\leq  \varepsilon\, |u_{bh}|_{1,{\Omega}}
 |z_h|_{1,{\Omega}} +   \| \bb\|_{\infty, \Omega} \|u_{bh}\|_{1,\Omega} \| z_h\|_{0,\Omega} +C\sigma\,\|u_{bh}\|_{0,\Omega}^{}\|z_h\|_{0,\Omega}^{} \nonumber\\
 &\leq C \|u_{bh}\|_{1,{\Omega}} | z_h|_{1,\Omega}\, .
 \end{align}
 In addition, using the mesh regularity, $\alpha_E^{}(\cdot)\le 1$, and the local trace inequality, we 
arrive at
  \begin{align}
  |d_h(z_h + u_{bh} ; u_ {bh}, z_h)  |&= \sum_{E\in \calE_h}\,  \gamma_0^{}\,  h_E^2 \,  \alpha_E^{}(z_h + u_{bh}) (\partial_{\bt} u_{bh} , \partial_{\bt} z_h)_E   \nonumber  \\
&\leq \sum_{E\in \calE_h}\, \gamma_0^{}\,  h_E^2\,  \|\partial_{\bt} u_{bh}\|_{0,E} \|\partial_{\bt} z_h \|_{0,E} \nonumber\\
& \leq   C  \, h \,|u_{bh}|_{1,\Omega} | z_h|_{1,\Omega} \, .
\end{align}
We can thus conclude that
 \begin{equation*}
[T_h z_h, z_h] \ge \varepsilon |z_h|_{1,\Omega}^2 - C\,\|u_{bh}\|_{1,\Omega} | z_h|_{1,\Omega} -
 \|f\|_{0,{\Omega}} \|z_h\|_{0,\Omega}\,.
\end{equation*} 
The claimed result arises by applying the Poincar\'e and Young inequalities to the last relation. \qed
\end{proof}

The solvability of the nonlinear problem \eqref{LSC} appears as a consequence of the above
result and Brower's fixed point Theorem.

\begin{theorem}
The  discrete problem \eqref{LSC} has at least one solution. Moreover, if $C_{\rm lip}\gamma_0^{}\,h < \varepsilon$, where $C_{\rm lip}$ is the constant from Lemma
\ref{Lemma-Lip-2}, then the
solution is unique.
\end{theorem}
\begin{proof}  First, since the bilinear form $a(\cdot,\cdot)$ is continuous, and $d_h(\cdot;\cdot,\cdot)$ is
Lipschitz continuous, then the operator $T_h$ is Lipschitz continuous. 
 Next, in view of  \eqref{T_h_property}, fow any $z_h\in \mathcal{V}_h^0$  such that 
\[
|z_h|^2_{1,\Omega} = 2\frac{c_2(\|u_{bh}\|_{1,\Omega}^2+\|f\|_{0,\Omega}^2)}{c_1}\,,
\]
Lemma \ref{exis-1} gives
\begin{equation}
[T_h z_h, z_h] =  c_2(\|u_{bh}\|_{0,\Omega}^2 + \|f\|_{0,\Omega}^2) >0.
\end{equation} 
Then, from a consequence of Brower's fixed point Theorem (see \cite[Corollary 1.1, Ch. IV]{GR86}), there exists $\tilde{v}_h\in\mathcal{V}_h^0$ such that $T_h(\tilde{v}_h)=0$. Hence, $u_h:=\tilde{v}_h+u_{bh}$ solves
\eqref{LSC}. 

In order to prove uniqueness, let $u_h^1, u_h^2$ be two solutions of \eqref{LSC}. Then, using \eqref{LSC}
for both solutions, denoting $\tilde{e}^{}_h:=u_h^1-u_h^2$, and using the Lipschitz continuity of $d_h(\cdot;\cdot,\cdot)$, we 
obtain
\begin{align}
\varepsilon\,|\tilde{e}^{}_h|^2_{1,\Omega} 
\le a(\tilde{e}^{}_h,\tilde{e}^{}_h) &= -d_h(u_h^1;u_h^1,\tilde{e}^{}_h)+d_h(u_h^2;u_h^2,\tilde{e}^{}_h)\nonumber \\
&\le C_{\rm lip}\gamma_0^{}h\,|\tilde{e}^{}_h|^2_{1,\Omega}\,.
\end{align}
This leads to
\begin{equation}
\big(\varepsilon-C_{\rm lip}\gamma_0^{}h\big)\,|\tilde{e}^{}_h|^2_{1,\Omega} \le 0\,,
\end{equation}
which, using that $\tilde{e}^{}_h\in H^1_0(\Omega)$,  finishes the proof. \qed
\end{proof}

\subsection{The discrete maximum principle}

This section is devoted to prove that Method \eqref{LSC} preserves positivity. For this, we will
impose the following geometric hypothesis on the mesh. This hypothesis can be tracked back to \cite{XZ99},
and in two space dimensions it reduces to impose that the mesh is Delaunay.

\begin{assumption}\label{XZ-assumption}[Hypothesis of Xu and Zikatanov, cf. \cite{XZ99}]
For every internal edge $E\in \calE_h$ with end points $x_i$ and $x_j$
 the following inequality holds
\begin{equation}
\frac{1}{d(d-1)}\sum_{K\in\omega_E}|\omega_{ij}^K|\cot (\theta_{ij}^K)\ge 0\,,
\end{equation}
where $\theta_{ij}^K$ is the angle between the two facets in $K$ opposite to $x_i$ and $x_j$
(denoted by $F_{i,K}$ and $F_{j,K}$, respectively), and
$\omega^K_{ij}$ is the $(d-2)$-dimensional simplex $F_{i,K} \cap F_{j,K}$ opposite to the edge $E$.
\end{assumption}

We now introduce the discrete analogous of the maximum principle. This definition is very close to
the one from \cite{BE05}, and it leads to results which are, essentially, identical to those from that reference.

\begin{definition}[DMP] 
The semilinear form $\tilde{a}(\cdot;\cdot)$ is said to satisfy the {\it strong DMP property}
if the following holds: For all $u_h\in \mathcal{V}_h$ and for all interior vertices $x_i$, if $u_h$ is locally
minimal (resp. maximal) on the vertex $x_i$ over the macro-element $\Omega_i$, then there exist
negative quantites $( c_E)_{E\in \calE_i}$ such that
\begin{equation}
\tilde{a}(u_h;\psi_i)\le \sum_{E\in \calE_i} c_E\big|\partial_{\bt}u_h|_{E}\big|\,,
\end{equation}
(resp. $\tilde{a}(u_h;\psi_i)\ge -\sum_{E\in\calE_i} c_E\big|\partial_{\bt}u_h|_{E}\big|$).
Furthermore, we will say that the semilinear form satisfies the weak DMP property, if
the local minimum above is supposed to be negative (resp. the local maximum is supposed
positive).
\end{definition}

A direct consequence of this definition is the following result, whose proof is completely analogous
to the one from \cite[Proposition 2.5]{BE05}.

\begin{lemma}
Assume that the semilinear form $\tilde{a}(\cdot;\cdot)$ satisfies the DMP property. Assume
that $u_h\in \mathcal{V}_h$ solves \eqref{LSC} and that $f\ge 0$. Then, $u_h$ reaches
its minimum on the boundary $\partial\Omega$.
\end{lemma}

The following result states the DMP for \eqref{LSC}.

\begin{theorem}
Let us suppose that the mesh $\calT_h$ satisfies Assumption \ref{XZ-assumption}, and that
the parameter $\gamma_0^{}$ is large enough.
Then, the semilinar form $\tilde{a}(\cdot;\cdot)$ satisfies the  DMP property.
\end{theorem}

\begin{proof}
Let us suppose that $u_h$ has a negative local minimum at an interior node $x_i$. Then,
 $\alpha_E^{}(u_h)=1$ for all $E\in \calE_i$, which gives
\begin{equation}
\tilde{a}(u_h;\psi_i)= (\sigma u_h,\psi_i)_\Omega +\varepsilon(\nabla u_h,\nabla\psi_i)_\Omega+
(\bb\cdot\nabla u_h,\psi_i)_\Omega+\sum_{E\in \calE_i}\gamma_0^{}h_E^2(\partial_{\bt}u_h,\partial_{\bt} \psi_i)_E\,.
\end{equation}
We will analyse the expression above term-by-term. First, if
$u_h\le 0$ in the support of $\psi_i$, then $(\sigma
u_h,\psi_i)_\Omega \le 0$. Let us suppose now that $u_h$ changes sign in the support of $\psi_i$,
and let $K\in\Omega_i$ be an element in which $u_h$ changes sign. Let $x_j$ and $x_k$ be the other
nodes of $K$, $E_{ij}$ and $E_{ik}$ the correspoding edges, and $\tilde{x}_{ij}\in E_{ij}, \tilde{x}_{ik}\in E_{ik}$
the points in those edges in which $u_h$ vanishes (if $u_h$ vanishes in only one edge, then the proof is
analogous). Then, using
a quadrature formula and a Taylor expansion we obtain
\begin{align*}
(\sigma u_h,\psi_i)_K&=\sigma\frac{K}{6}\left(u_h\left(\frac{x_i+x_j}{2}\right)+u_h\left(\frac{x_i+x_k}{2}\right)\right)\\
&= \sigma\frac{K}{6}\left(u_h\left(\frac{x_i+x_j}{2}\right)-u_h(\tilde{x}_{ij})+u_h\left(\frac{x_i+x_k}{2}\right)-u_h(\tilde{x}_{ik})\right)\\
&\le \sigma\frac{K}{6}\big( h_{E_{ij}}^{}\big|\partial_{\bt}u_h|_{E_{ij}^{}}\big|+ h_{E_{ik}}^{}\big|\partial_{\bt}u_h|_{E_{ik}^{}}\big|\big)\,,
\end{align*}
and, adding up over all $K\in\Omega_i$ and using the mesh regularity we obtain
\begin{equation}
(\sigma u_h,\psi_i)_\Omega \leq C_0 \sigma  \sum_{E \in \mathcal{E}_i} h^3_E
|\partial_{\bf t} u_h|_E|\,.
\end{equation}
Also,  as in \cite{BE05} (see also \cite{RST08}), Assumption \ref{XZ-assumption} on the mesh leads to
\begin{equation}
\varepsilon (\nabla u_h,\nabla \psi_i)_\Omega\le 0\,.
\end{equation}
Moreover $\sum_{j=1}^N\psi_j=1$ gives $\sum_{j\in S_i} (\bb\cdot\nabla\psi_j,\psi_i)_\Omega=0$, and then
\begin{align}
(\bb\cdot\nabla u_h, \psi_i)_\Omega &= \sum_{j\in S_i}(\bb\cdot\nabla \psi_j, \psi_i)_\Omega u_h(x_j)
+(\bb\cdot\nabla \psi_i, \psi_i)_\Omega u_h(x_i) \nonumber\\
&= \sum_{j\in S_i}(\bb\cdot\nabla \psi_j, \psi_i)_\Omega \big( u_h(x_j)-u_h(x_i)\big) \nonumber\\
&= \sum_{E\in \calE_i}(\bb\cdot\nabla \psi_j, \psi_i)_\Omega h_E \big|\partial_{\bt} u_h|_E\big|\,, \label{229}
\end{align}
which, using the mesh regularity gives
\begin{equation}
(\bb\cdot\nabla u_h, \psi_i)_\Omega\le \sum_{E\in \calE_i} C_1\|\bb\|_{\infty,E}h_E^2|\partial_{\bt} u_h|_E|\,.
\end{equation}
Finally, since $u_h(x_i)$ is a local minimum, then in every $E\in \calE_i$, $\partial_{\bt} u_h$ and 
$\partial_{\bt} \psi_i$
have different signs (independently of the orientation of the tangential vector in $E$), which gives
\begin{equation}
\sum_{E\in \calE_i}\gamma_0^{}h_E^2(\partial_{\bt}u_h,\partial_{\bt} \psi_i)_E = 
-\sum_{E\in \calE_i}\gamma_0^{}h_E^2\big|\partial_{\bt} u_h|_E\big|\,.
\end{equation}

Hence, gathering all the above computations, we arrive at
\begin{equation}\label{232}
\tilde{a}(u_h;\psi_i)\le -\sum_{E\in \calE_i}(\gamma_0^{}- C_0 \sigma h_E-
C_1\|\bb\|_{\infty,E})h_E^2\big|\partial_{\bt} u_h|_E\big|\,,
\end{equation}
and the result follows assuming that $\gamma_0^{}> C_0 \sigma h_E + C_1\|\bb\|_{\infty,E}$. Finally, we notice that if $\sigma=0$
then the sign of the strict minimum is irrelevant. \qed
\end{proof}

\begin{remark} It is interesting to remark that the hypothesis on the meshes of the triangulation
can be avoided if the problem is supposed to be strongly convection-dominated. In fact, 
following analogous steps to those used to prove \eqref{229} we can arrive at 
\begin{equation}
\varepsilon(\nabla u_h,\nabla\psi_i)_\Omega = \varepsilon\sum_{E\in \calE_i}(\nabla\psi_j,\nabla\psi_i)_\Omega
h_E|\partial_{\bt}u_h| \le \sum_{E\in \calE_i}C_1\varepsilon h_E |\partial_{\bt}u_h|\,.
\end{equation}
Replacing this into the steps leading to \eqref{232} gives
\begin{equation}
\tilde{a}(u_h;\psi_i)\le -\sum_{E\in \calE_i}(\gamma_0^{}-C_0 \sigma h_E - C_1\|\bb\|_{\infty,E} - C_2\varepsilon h_E^{-1})h_E^2|\partial_{\bt} u_h|\,,
\end{equation}
and the proof follows by assuming that $\gamma_0^{}\ge C_0 \sigma h_E  +C_1\|\bb\|_{\infty,E} + C_2\varepsilon h_E^{-1}$.

The last result is only interesting if $\varepsilon h_E^{-1}$ stays bounded, which means this is
applicable only in the case the problem is highly convection-dominated. In this sense, the method proposed in this
work can be applied to scalar conservation laws, regardless of the geometrical impositions on the
mesh.  Similar results have been obtained recently in \cite{GN14CMAME,GNPY14SINUM}.
\end{remark}

\section{Convergence}\label{error-estimates}

The error will be  analysed using the following mesh-dependent norm:
\begin{equation}\label{norm}
\|v_h\|_h^2:= \sigma\|v_h\|^2_{0,\Omega}+\varepsilon\,|v_h|^2_{1,\Omega}+d_h(u_h;v_h,v_h)\,.
\end{equation}
As usual, the error $e:=u-u_h$ is  split as follows
\begin{equation}
e = u- u_h = (u- i_h^{}u) + (i_h^{} u - u_h):= \rho_h + e_h,
\end{equation}
where $i_h^{}:C^0(\overline{\Omega})\cap H^1_0(\Omega)\to\mathcal{V}^0_h$ stands for the Cl\'ement interpolation operator.
Using standard interpolation estimates (see \cite{EG04}), the fact that $\alpha_E^{}(\cdot)\le 1$, 
and the mesh regularity,  the following 
bound for $\rho_h$ follows:
\begin{equation}\label{interpol}
\|\rho_h\|_h^{}
  \le C( \varepsilon^{\frac{1}{2}} + \sigma^{\frac{1}{2}} h + \gamma_0^{}h^{\frac{1}{2}} )\,h\, \|u\|_{2,\Omega}.
\end{equation}

It only remains to estimate the discrete error. This is done in the next result.

\begin{lemma}\label{discrete-error}
Let us suppose $u\in H^2(\Omega)\cap H^1_0(\Omega)$.
 Then, there exists $C>0$, independent of $h$ and $\varepsilon$,  such that
\begin{equation}
\|e_h\|_h^{}
 \le C\,\big(\varepsilon+\sigma^{-1}\{\|\bb\|^2_{\infty,\Omega}+\sigma^2\}\big)^{\frac{1}{2}}h\|u\|_{2,\Omega}+Ch^{\frac{1}{2}}\,\|u\|_{1,\Omega}\,.
\end{equation}
\end{lemma}
\begin{proof}
First, from the definition of $a$ and $d_h$ we get
\begin{align}
\|e_h\|^2_h =&\, a(e_h,e_h)+d_h(u_h;e_h,e_h) \nonumber\\
=& \,a(i_h^{}u,e_h)-\big\{ a(u_h,e_h)+d_h(u_h;u_h,e_h)\big\}+d_h(u_h;i_h^{}u,e_h)  \nonumber\\
=& \,-a(\rho_h,e_h)+d_h(u_h;i_h^{}u,e_h)\,. \label{error-1}
\end{align}
Next,  the continuity of $a$ gives
\begin{align}
a(\rho_h,e_h) &\leq  
(\sigma\|\rho_h\|^2_{0,\Omega}+[\varepsilon+\sigma^{-1}\|\bb\|_{\infty,\Omega}^2]
\,|\rho_h|^2_{1,\Omega})^{\frac12}
\|e_h\|_h\, \nonumber \\
&\leq C( \varepsilon^{\frac{1}{2}} + \sigma^{-1/2} \|\bb\|_{\infty,\Omega}+\sigma^{\frac{1}{2}} h )\,h\, \|u\|_{2,\Omega}\|e_h\|_h\,.\label{error-2}
\end{align}
Moreover, since $d_h(u_h;\cdot,\cdot)$ is a symmetric positive semi-definite bilinear form it satisfies
Cauchy-Schwarz's inequality, which gives
\begin{equation}\label{error-3}
d_h(u_h;i_h^{}u,e_h)\le d_h(u_h;i_h^{}u,i_h^{}u)^{\frac{1}{2}}d_h(u_h;e_h,e_h)^{\frac{1}{2}}\le 
d_h(u_h;i_h^{}u,i_h^{}u)^{\frac{1}{2}} \|e_h\|_h\,.
\end{equation}
Then, inserting \eqref{error-2} and \eqref{error-3} into \eqref{error-1}, and using Young's inequality,
we arrive at
\begin{equation}\label{error-4}
\|e_h\|^2_h \le C( \varepsilon^{\frac{1}{2}} + \sigma^{-1/2} \|\bb\|_{\infty,\Omega}+\sigma^{\frac{1}{2}} h )^2\,h^2\, \|u\|^2_{2,\Omega} +C\,d_h(u_h;i_h^{}u,i_h^{}u)\,.
\end{equation}
It only remains to bound the consistency error $d_h(u_h;i_h^{}u,i_h^{}u)$ in \eqref{error-4}.
The definition of $d_h(\cdot;\cdot,\cdot)$, $\alpha_E^{}(u_h)\le 1$,
a local trace inequality, the mesh regularity, and the $H^1(\Omega)$-stability of $i_h^{}$, give
\begin{equation}
d_h(u_h;i_h^{}u,i_h^{}u) =\sum_{E\in\calE_h}\gamma_0^{}h_E^2\alpha_E^{}(u_h)\|\partial_{\bt}i_h^{}u\|_{0,E}^2 
\le \gamma_0^{}h\sum_{E\in\calE_h}h_E\,\|\partial_{\bt}i_h^{}u\|^2_{0,E}\le Ch\,\|u\|^2_{1,\Omega}\,. \label{error-5}
\end{equation}
Then, the result arises inserting \eqref{error-5} into \eqref{error-4}. \qed
\end{proof}

Collecting \eqref{interpol} and Lemma \ref{discrete-error} we then obtain the following error estimate for \eqref{LSC}.

\begin{theorem}
Let us suppose $u\in H^2(\Omega)\cap H^1_0(\Omega)$.
 Then, there exists $C>0$, independent of $h$ and $\varepsilon$,  such that
\begin{equation}
\|e\|_h^{}
 \le C\,\big(\varepsilon+\sigma^{-1}\{\|\bb\|^2_{\infty,\Omega}+\sigma^2\}\big)^{\frac{1}{2}}h\|u\|_{2,\Omega}+Ch^{\frac{1}{2}}\,\|u\|_{1,\Omega}\,.
\end{equation}
\end{theorem}

The following result states that for meshes which are symmetric with respect to their interior nodes, 
the method converges with a higher order. This result's main interest lies in the diffusion dominated
regime, due to the factor $\varepsilon^{-\frac{1}{2}}$ present in  the estimate. The combination
of Lipschitz continuity and linearity preservation seems to be novel, and that is why we do detail
it now.

\begin{theorem}\label{theo-optimal}
Let us suppose $u\in H^2(\Omega)\cap H^1_0(\Omega)$ and that the mesh is symmetric with respect to
its internal nodes.
 Then, there exists $C>0$, independent of $h$ and $\varepsilon$,  such that
\begin{equation}
\|e\|_h^{}
 \le C\,\big(\varepsilon+\sigma^{-1}\{\|\bb\|^2_{\infty,\Omega}+\sigma^2\}\big)^{\frac{1}{2}}h\|u\|_{2,\Omega}+C\frac{h}{\sqrt{\varepsilon}}\,
\|u\|_{1,\Omega}\,.
\end{equation}
\end{theorem}

\begin{proof} It is enough to bound  the consistency error $d(u_h;i_h^{}u,i_h^{}u)$. We have
\begin{align}
d_h(u_h;i_h^{}u,i_h^{}u)&=\big\{ d_h(u_h;i_h^{}u,i_h^{}u)-d_h(i_h^{}u;i_h^{}u,i_h^{}u)\big\}+d_h(i_h^{}u;i_h^{}u,i_h^{}u)\nonumber\\
&=:\mathrm{I}+\mathrm{II}\,. \label{optimal-1}
\end{align}
The first term is bounded as in the proof of Lemma \ref{Lemma-Lip-2}. In fact, in that proof,
the bound for the second term in \eqref{Lip-1} leads to the following
\begin{align}
\mathrm{I}&=\sum_{E\in\calE_h}(\alpha_E^{}(u_h)-\alpha_E(i_h^{}u))\gamma_0^{}h_E^2(\partial_{\bt} i_h^{}u,\partial_{\bt} i_h^{}u)_E\nonumber\\
&\le Ch|u_h-i_h^{}u|_{1,\Omega}^{}|i_h^{}u|^{}_{1,\Omega}\nonumber\\
&\le \frac{\varepsilon}{2}\,|u_h-i_h^{}u|_{1,\Omega}^2+C\frac{h^2}{\varepsilon}\,\|u\|^2_{1,\Omega}\,,\label{optimal-2}
\end{align}
where we have also used the $H^1(\Omega)$-stability of $i_h^{}$.
To bound $\mathrm{II}$ we use the linearity preservation and the Lipschitz continuity of $d_h(\cdot;\cdot,\cdot)$.
More precisely, for a given $E\in\calE_h$ we introduce the function $i_E^{}u\in\PP_1(\omega_E)$ as the
unique solution of the problem
\begin{align}
(\nabla i_E^{}u,\nabla\psi)_{\omega_E^{}}&=(\nabla u,\nabla\psi)_{\omega_E^{}}\quad\forall\,\psi\in \PP_1(\omega_E^{})\,,\label{optimal-3a}\\
(i_E^{}u,1)_{\omega_E^{}}&=(u,1)_{\omega_E^{}}\,.\nonumber
\end{align}
Using standard finite element approximation results (see \cite{EG04}), $i_E^{}u$ satisfies
\begin{equation}
|u-i_E^{}u|^{}_{1,\omega_E^{}}\le Ch_E^{}|u|^{}_{2,\omega_E^{}}\,.\label{optimal-3b}
\end{equation}
Since the mesh is symmetric with respect to its internal nodes,  $\alpha_E^{}(i_E^{}u)=0$. Then,
proceeding as in the bound for $\mathrm{I}$ we obtain
\begin{align}
\mathrm{II} &= \sum_{E\in\calE_h} (\alpha_E^{}(i_h^{}u)-\alpha_E^{}(i_E^{}u))\gamma_0^{}h_E^2\,
(\partial_{\bt} i_h^{}u,\partial_{\bt} i_h^{}u)_E\nonumber\\
&\le Ch\,\left\{\sum_{E\in\calE_h}|i_h^{}u-i_E^{}u|^2_{1,\omega_E^{}}\right\}^{\frac{1}{2}}|i_h^{}u|_{1,\Omega}^{}\nonumber\\
&\le Ch^2|u|_{2,\Omega}^{}\|u\|_{1,\Omega}^{}\,.\label{optimal-4}
\end{align}
Then, inserting \eqref{optimal-2} and \eqref{optimal-4} into \eqref{optimal-1} we obtain
\begin{equation}\label{optimal-5}
d_h(u_h;i_h^{}u,i_h^{}u)\le \frac{\varepsilon}{2}\,|u_h-i_h^{}u|_{1,\Omega}^2+C\frac{h^2}{\varepsilon}\,\|u\|^2_{1,\Omega}+Ch^2|u|_{2,\Omega}^{}\|u\|_{1,\Omega}^{}\,,
\end{equation}
and the result follows by rearranging terms. \qed
\end{proof}

\section{A link to algebraic flux correction schemes}\label{link-AFC}

Method \eqref{LSC} has been presented having as motivation the study of the effect of adding
edge-based diffusion into the equations to impose the discrete maximum principle. Another 
family of methods that are built with the same purpose is the
AFC schemes.
This section is devoted to study the relationship between the two approaches, and that is why
we now summarise the main building principles of AFC schemes.

The starting point of an algebraic flux-correction scheme is a discretisation of the convection-diffusion-reaction
equation which leads to the linear system  
 \begin{equation}\label{Galerkin}
\mathbb{A}\rmU=\mathbb{G}\,,
\end{equation}
where $\mathbb{A}=(a_{ij})_{i,j=1}^N$, $\rmU=\{u_h(x_i)\}_{i=1}^N$ and $\mathbb{G}=\{g_i\}_{i=1}^N$.
 The first step of these schemes is to identify which parts of the system matrix $\mathbb{A}$ 
are responsible for the violation of the discrete maximum principle. To achieve this, the diffusion matrix
$\mathbb{D}=(d_{ij})_{i,j=1}^N$ is built, where
\[
d_{ij}=d_{ji}=-\max\{ a_{ij},0,a_{ji}\}\quad\forall i\not= j\,\quad d_{ii}=-\sum_{j\not= i} d_{ij}\,.
\]
Adding $\mathbb{D}\rmU$ both sides of \eqref{Galerkin} we obtain
\begin{equation}
\tilde{\mathbb{A}\rm}\rmU=\mathbb{G}+\mathbb{D}\rmU\,,
\end{equation}
where $\tilde{\mathbb{A}}:=\mathbb{A}+\mathbb{D}$. Since the matrix $\tilde{\mathbb{A}}$ fullfils the
hypothesis to guarantee the discrete maximum principle, then the oscillations
that appear in a non-stabilised discretisation \eqref{Galerkin} are due to the right-hand side.  
This is why the right-hand side is now rewritten. Using that the row-sums of $\mathbb{D}$ are zero, then
\[
(\mathbb{D}\rmU)_i=\sum_{j\not= i}f_{ij}\qquad\textrm{where}\quad f_{ij}=d_{ij}(u_h(x_j)-u_h(x_i))\,.
\]
The quantities $f_{ij}$ are called {\it fluxes}.
Then, the AFC schemes are based on introducing limiters $\alpha_{ij}(u_h)$ such that $\alpha_{ij}\in [0,1]$,
 $\alpha_{ij}=\alpha_{ji}$, and $\alpha_{ij}=1$ if $x_i$ and $x_j$ are both Dirichlet nodes. Then, after introducing
these limiters, the method reads as follows:
\begin{equation}\label{AFC-alge}
\mathbb{A}\rmU+\sum_{i, j=1}^N(1-\alpha_{ij}(u_h)) d_{ij}\,(u_h(x_j)-u_h(x_i))= g_i\,.
\end{equation}

The most popular limiters in practice are Zalesak's limiters (see, \cite{Zalesak79,Kuzmin06,Kuzmin07,Kuzmin12},
and the recent review \cite{Kuzmin-book} for examples). The analysis of this class of methods for
a class of limiters that includes the Zalesak one  has been carried out recently in \cite{BJK15,BJK15b}. In particular,
in \cite{BJK15b} an $O(h^{\frac{1}{2}})$ convergence rate was proved for the case in which the
mesh used satisfies Assumption \ref{XZ-assumption}. In the case of meshes that do not satisfy
this assumption, then no convergence can be proved. This result is optimal, as the numerical
results in \cite{BJK15b} show.

Following \cite{BJK15b}, \eqref{AFC-alge} can be written as the following weak problem:
Find $u_h\in\mathcal{V}_h$ such that $u_h-u_{bh}\in\mathcal{V}_h^0$, and 
\begin{equation}
a(u_h,v_h)+\tilde{d}_h(u_h;u_h,v_h)=(f,v_h)_\Omega\qquad\forall\, v_h\in\mathcal{V}_h^0\,,
\end{equation}
where the nonlinear form $\tilde{d}_h(\cdot;\cdot,\cdot)$ is given by
\begin{equation}
\tilde{d}_h(u_h;u_h,v_h)=\sum_{i, j=1}^N(1-\alpha_{ij}(u_h)) d_{ij}\,(u_h(x_j)-u_h(x_i))v_h(x_i)\,.
\end{equation}
Next, to link this to the method analysed in the last sections, we use the symmetry of $\mathbb{D}$, and of
the limiters $\alpha_{ij}=\alpha_{ji}$, and a simple calculation gives:
\begin{align}
\tilde{d}_h(u_h;u_h,v_h)&=\sum_{i>j}(1-\alpha_{ij}(u_h)) d_{ij}\,(u_h(x_j)-u_h(x_i))v_h(x_i) \nonumber\\
&\quad + \sum_{i<j}(1-\alpha_{ij}(u_h)) d_{ij}\,(u_h(x_j)-u_h(x_i))v_h(x_i) \nonumber\\
&= \sum_{i>j}(1-\alpha_{ij}(u_h)) d_{ij}\,(u_h(x_j)-u_h(x_i))v_h(x_i)\nonumber\\
&\quad + \sum_{i>j}(1-\alpha_{ji}(u_h)) d_{ji}\,(u_h(x_i)-u_h(x_j))v_h(x_j) \nonumber\\
&= \sum_{i>j}(1-\alpha_{ij}(u_h)) d_{ij}\,(u_h(x_j)-u_h(x_i))(v_h(x_i)-v_h(x_j))\,.
\end{align}
Then, since $d_{ij}=0$ for $j\not\in S_i$, $\tilde{d}_h(\cdot;\cdot,\cdot)$ can be rewritten as 
\begin{equation}\label{finalAFC}
\tilde{d}_h(u_h;u_h,v_h)
= \sum_{E\in \calE_h} (1-\alpha_{ij}(u_h)) |d_{ij}|h_E\,(\partial_{\bt}u_h,\partial_{\bt}v_h)_E\,,
\end{equation}
where we have adopted the convention that an edge $E\in\calE_h$ has endpoints $x_i$ and $x_j$,
and used that $\alpha_{ij}=1$ for edges included in the Dirichlet boundary.

Method \eqref{LSC} then appears as an algebraic flux-correction scheme, with a different definition of the
limiters. Indeed comparing \eqref{LSC} with \eqref{finalAFC} we get
the equivalent AFC scheme if we choose $\alpha_{ij}(u_h)$ such that
\[
(1-\alpha_{ij}(u_h)) |d_{ij}|h_E =  \gamma_0^{}\,  h_E^d \, \alpha_E^{}(w_h).
\]
The new definition of the limiters made it possible to write some convergence and existence
results, also present in \cite{BJK15b}, in a more precise way, and improve in some of them. In particular, the new limiters
make it possible to prove convergence for general meshes, as well as to prove uniqueness of
solutions and optimal convergence in the diffusion dominated regime.

\section{Numerical Results}\label{numerics}

In this section we present three sets of numerical results for bi-dimensional problems. The nonlinear system
\eqref{LSC} has been solved using the following fixed-point algorithm with damping: Starting with the Galerkin
solution $u_h^0$, then compute a sequence $\{ u_h^k\}$ defined by
\begin{equation}
u_h^{k+1}=u_h^k+\omega\,(\tilde{u}_h^{k+1}-u_h^k)\quad k=0,1,2,\ldots\,,
\end{equation}
where $\omega\in (0,1)$ is a damping parameter, and $\tilde{u}_h^{k+1}$ solves: $\tilde{u}_h^{k+1}-u_{bh}\in
\mathcal{V}_h^0$, and 
\begin{equation}\label{linearised}
a(\tilde{u}_h^{k+1},v_h)+d_h(u_h^k;\tilde{u}_h^{k+1},v_h)=(f,v_h)\qquad\forall\, v_h\in \mathcal{V}_h^0\,.
\end{equation}
In all our calculations we have used $\omega=0.1$, and stopped the iterations when the residual 
$\boldsymbol{R}^{k}:=\left(\tilde{a}(u_h^{k+1};\psi_i)-(f,\psi_i)_\Omega\right)_{i=1,\ldots,{\rm dim}(\mathcal{V}_h^0)}$ has an euclidean norm smaller than, or equal to, $10^{-8}$. 

\subsection{Convergence for a smooth solution}

We take $ \bb=(2,1)$, $\sigma=1$, and different values for $\varepsilon$. We have selected the right-hand-side 
and boundary  conditions in such a way that
the solution is given by $u(x,y)=\sin(2\pi x)\sin(2\pi y)$. The meshes used were the three-directional mesh (c)
and the non-Delaunay mesh~(d) in Figure \ref{fig-mesh}. In these calculations we have used $\gamma_0^{}=3$ and
$p=4$.

The results in Tables \ref{convdom-meshc}-\ref{diffdom-meshd} match the theoretical results. In particular we observe a first order convergence in the diffusion-dominated regime for the Mesh~(c), 
as predicted by Theorem \ref{theo-optimal}, and a second order convergence
in the $L^2$ norm of the error for both the convection and diffusion-dominated regimes. The latter is in accordance
with the empirical observations that linearity preservation implies such a convergence. For Mesh~(d),
which is non-symmetric, and hence the method is no longer linearity preserving, we can observe a first
order convergence in both regimes. This convergence is not affected by the non-Delaunay character of
the mesh. 

\begin{table}[htb]
\begin{center}
\caption{$\varepsilon=10^{-6}$, numerical results for Grid~(c).}
\label{convdom-meshc} 
\begin{tabular}{c|rrrrrr}
$l$ & $\|u-u_h\|_{0,\Omega}$ & ord. &   $|u-u_h|_{1,\Omega}$ & ord.& 
 $\|u-u_h\|_h$ & ord.\\ \hline
3 & 0.49391  & --      &  4.38896  & --       & 3.62380 & -- \\
4 & 0.47965  & 0.04  &  4.26871 &  0.04  &   3.08479 & 0.23  \\
5 & 0.19110  & 1.33  &  2.71665 &  0.65  &   1.08371  & 1.51\\
6 & 0.04080  & 2.23  &  1.55469 & 0.81  &     0.22671 &  2.26  \\
7 & 0.00683  & 2.58  &  0.64692 & 1.27  &  0.03904 &  2.54   \\
8 & 0.00119  & 2.52  &  0.27480 & 1.24  &  0.00689 & 2.50
\end{tabular}
\end{center}
\end{table}

\begin{table}[htb]
\begin{center}
\caption{$\varepsilon=1$, numerical results for Grid~(c).}
\label{diffdom-meshc} 
\begin{tabular}{c|rrrrrr}
$l$ & $\|u-u_h\|_{0,\Omega}$ & ord. &   $|u-u_h|_{1,\Omega}$ & ord.& 
 $\|u-u_h\|_h$ & ord.\\ \hline
3 & 0.38594   & --     &    3.48242  & -- &     5.44504  & -- \\
4 & 0.16557  & 1.22  &   1.90920  & 0.87  &   2.26966 &  1.26  \\
5 & 0.03268  & 2.34  &   0.89029  & 1.10  &   0.92785 &  1.29 \\
6 & 0.00612  & 2.42  &   0.43637  & 1.03  &   0.43912  & 1.08 \\
7 & 0.00141  & 2.12  &   0.21800  & 1.00  &   0.21818 &  1.01   \\
8 & 0.00035  & 2.02  &   0.10903  & 1.00   &  0.10904 &  1.00
\end{tabular}
\end{center}
\end{table}

\begin{table}[htb]
\begin{center}
\caption{$\varepsilon=10^{-6}$, numerical results for Grid~(d).}
\label{convdom-meshd} 
\begin{tabular}{c|rrrrrr}
$l$ & $\|u-u_h\|_{0,\Omega}$ & ord. &   $|u-u_h|_{1,\Omega}$ & ord.& 
 $\|u-u_h\|_h$ & ord.\\ \hline
3&  0.48754    & -- &   4.33607   & -- &    5.06989  & -- \\
4& 0.45680  & 0.09   & 4.11426  & 0.08    &2.93242  & 0.79\\
5& 0.17080  & 1.42   & 3.15455  & 0.38    & 1.05213  & 1.48 \\
6& 0.04330  & 1.98    & 2.23948  & 0.49   &  0.26065  & 2.01 \\
7 & 0.01165  & 1.89    & 1.72410  & 0.38   & 0.05482  & 2.25  \\
8 & 0.00474  & 1.30    & 1.63424  & 0.08    & 0.02087  & 1.39
\end{tabular}
\end{center}
\end{table}

\begin{table}[htb]
\begin{center}
\caption{$\varepsilon=1$, numerical results for Grid~(d).}
\label{diffdom-meshd} 
\begin{tabular}{c|rrrrrr}
$l$ & $\|u-u_h\|_{0,\Omega}$ & ord. &   $|u-u_h|_{1,\Omega}$ & ord.& 
 $\|u-u_h\|_h$ & ord.\\ \hline
3 & 0.38351 & -- &  3.52996  & -- & 5.57464 & -- \\
4 & 0.16616  & 1.21    & 2.00539  & 0.82    & 2.41681  & 1.21 \\
5 & 0.04513  & 1.88    & 0.98086  & 1.03    & 1.03172  & 1.23 \\
6 & 0.01277  & 1.82    & 0.48118  & 1.03    & 0.48720  & 1.08 \\
7 & 0.00423  & 1.59    & 0.23973  & 1.01    & 0.24059  & 1.02   \\
8 & 0.00163  & 1.38    & 0.11982  & 1.00    & 0.11998  & 1.00
\end{tabular}
\end{center}
\end{table}

\subsection{A problem with one inner layer, and a rotating convective field}
We use $\varepsilon =10^{-5}$, $f=0$, $\sigma=0$, $\bb=(-y,x)$, homogeneous Neumann boundary conditions
on exit, and 
\[
g(x,y)=\left\{ \begin{array}{cl} 1 & \textrm{if}\; x\le 0.5\,,\\
0 & \textrm{else}\,, \end{array}\right. \]
as Dirichlet condition at entry. We have solved this problem on a uniform refinement of the 
three-directional from Mesh~(c) in Figure \ref{fig-mesh}. 
The parameter $\gamma_0^{}$ has been set to $1$, and the results show no violation of the DMP.
The results for this case are depicted in Figure  \ref{Fig2}. We can observe that
the increase in the value of $p$ provides a solution whose inner layer is much sharper than the
choice $p=1$.

\begin{figure}[h!tb]
\begin{center}
\subfigure[]{\includegraphics[scale=0.35]{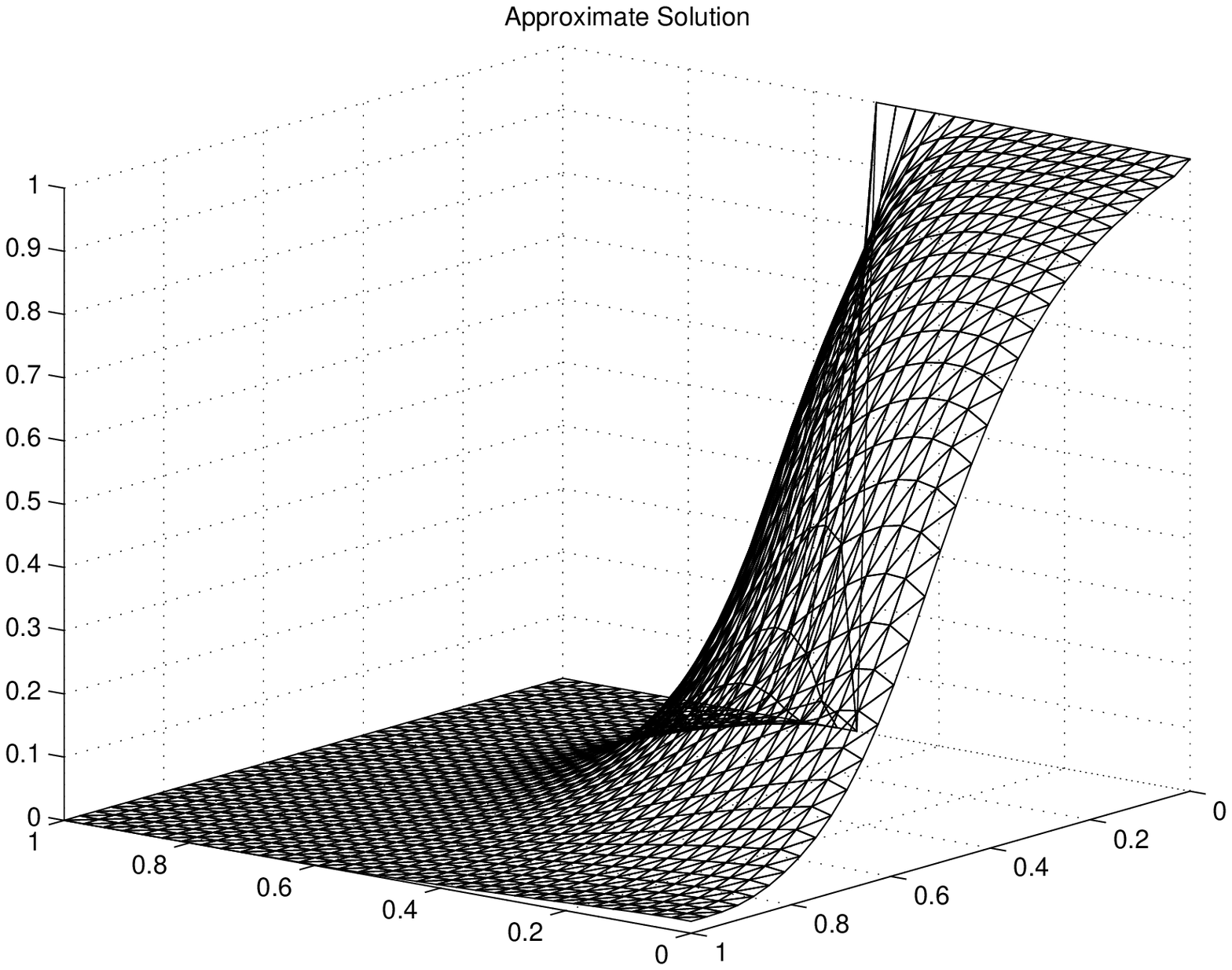}}
\subfigure[]{\includegraphics[scale=0.35]{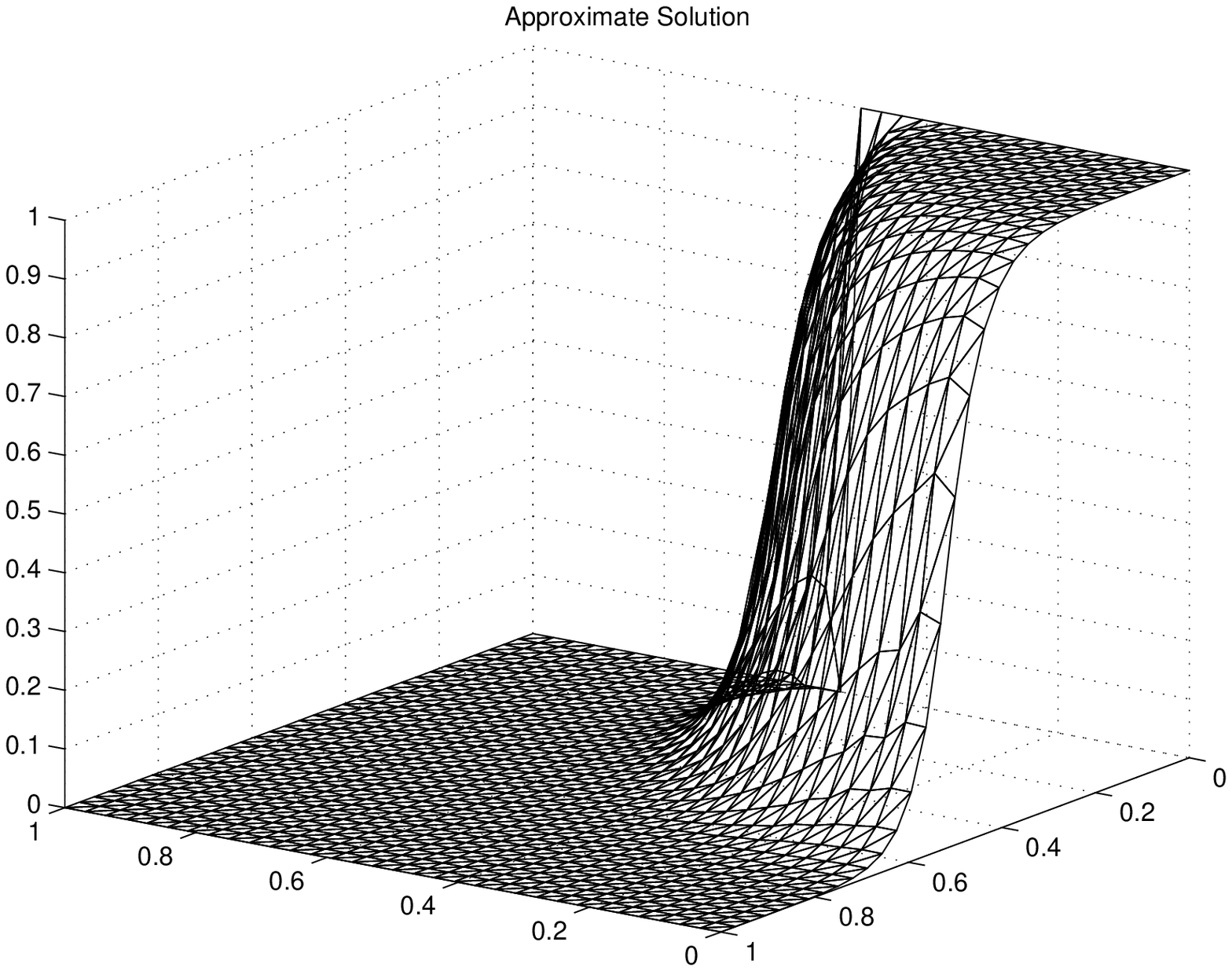}}
\end{center}
\caption{Discrete solution  for $p=1$ (left) and $p=4$ (right). }
\label{Fig2}
\end{figure}

\subsection{Advection skew to the mesh}
We use $\varepsilon =10^{-5}$, $f=0$, $\sigma=0$ $\bb=\left(\cos\left(\frac{\pi}{3}\right),\sin\left(\frac{\pi}{3}\right)\right)$, and 
\[
g(x,y)=\left\{ \begin{array}{cl} 1 & \textrm{if}\; x\,=\, 0 \; \textrm{or}\; y=1\,,\\
0 & \textrm{else}\,, \end{array}\right. \]
as Dirichlet condition. We have solved this problem on a criss-cross mesh as shown in Mesh~(a)
in Figure \ref{fig-mesh}. We have used the parameter $\gamma_0^{}=0.75$, and, again, no violations
of the DMP have been observed.
The results are depicted in Figure \ref{Fig3}, where we can observe  much 
sharper layers (especially the internal one) when $p=4$ is used.

\begin{figure}[h!tb]
\begin{center}
\subfigure[]{\includegraphics[scale=0.35]{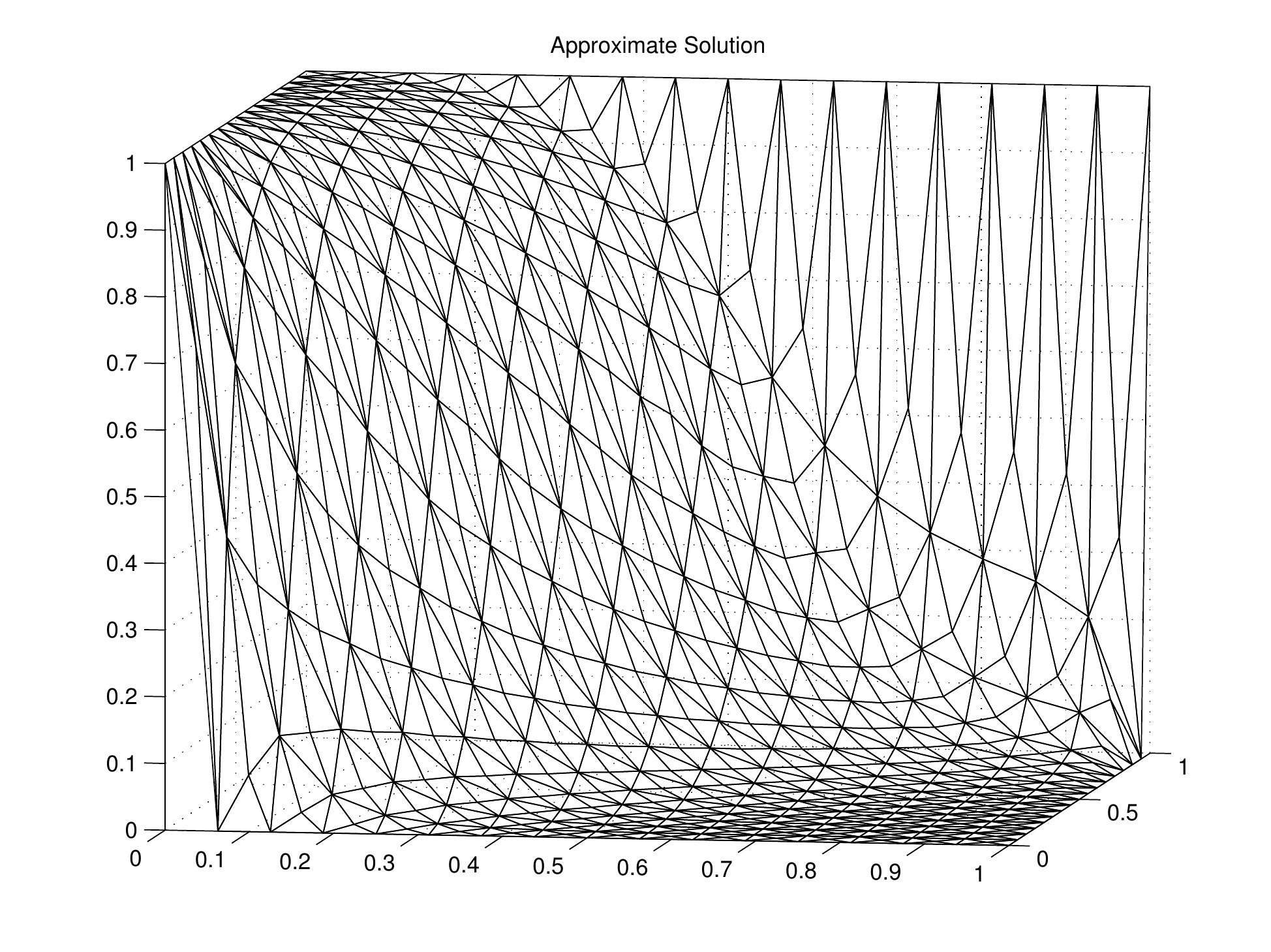}}
\subfigure[]{\includegraphics[scale=0.35]{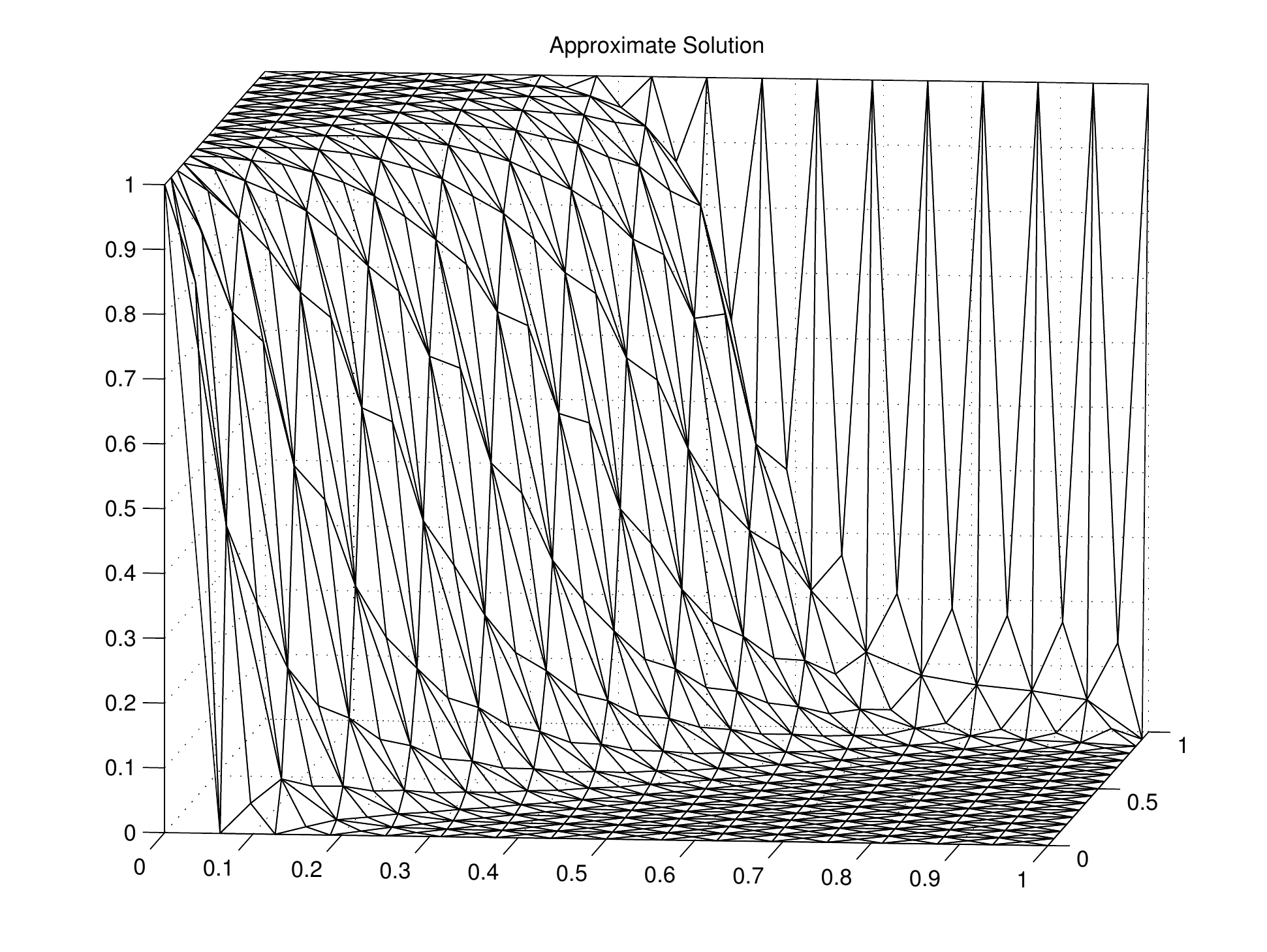}}
\end{center}
\caption{Discrete solution  for $p=1$ (left) and $p=4$ (right). }
\label{Fig3}
\end{figure}

\section*{Acknowledgements} The work of GRB and FK has been partially funded by the Leverhulme Trust
via the Research Project Grant No. RPG-2012-483. The authors would like to thank Volker John and Petr Knobloch  
for very helpful discussions.

\bibliographystyle{spmpsci} 
\bibliography{BBK}

\end{document}